\documentclass[a4paper,12pt]{amsart}
\usepackage[T1]{fontenc}
\usepackage[centering]{geometry}
\usepackage[latin1]{inputenc}
\usepackage{amsthm, amsmath}   
\usepackage{latexsym,amssymb}
\usepackage{algorithmic}
\usepackage{amssymb}
\usepackage{graphicx,subfigure}
\usepackage{times}
\usepackage{enumerate}
\usepackage[usenames,dvipsnames]{pstricks}
\usepackage{epsfig}
\usepackage{pst-node}
\usepackage{color}

\newcommand{\N}{\mathbb{N}}

\newcommand{\Z}{\mathbb{Z}}

\newcommand{\ch}{\mathrm{Chomp}}

\newcommand{\mC}{\mathcal{C}}

\newcommand{\mK}{\mathcal{K}}

\newcommand{\KG}{\mathcal{KG}}
\newcommand{\J}{\mathcal{J}}

\theoremstyle{plain}
\newtheorem{theorem}{Theorem}[section]
\newtheorem{lemma}[theorem]{Lemma}
\newtheorem{corollary}[theorem]{Corollary}
\newtheorem{proposition}[theorem]{Proposition}

\theoremstyle{definition}
\newtheorem{definition}[theorem]{Definition}

\newtheorem{remark}[theorem]{Remark}

\title{Chomp on generalized Kneser graphs and others}
\author[I. Garc\'{i}a-Marco]{Ignacio Garc\'{i}a-Marco}
\author[K. Knauer]{Kolja Knauer}
\author[L. P. Montejano]{Luis Pedro Montejano}

\address{Facultad de Ciencias, Universidad de La Laguna. La Laguna, Spain.}
\email{iggarcia@ull.es}
\address{Aix Marseille Univ, Universit\'e de Toulon, CNRS, LIS, Marseille, France}
\email{kolja.knauer@lis-lab.fr}
\address{CONACYT Research Fellow - Centro de Investigaci\'on en Matem\'aticas, Guanajuato, GTO, M\'exico}
\email{luis.montejano@cimat.mx}

\begin{document}

\begin{abstract}
In chomp on graphs, two players alternatingly pick an edge or a vertex from a graph. The player that cannot move any more loses.
The questions one wants to answer for a given graph are: Which player has a winning strategy? Can a explicit strategy be devised?
We answer these questions (and determine the Nim-value) for the class of generalized Kneser graphs and for several families of Johnson graphs. We also generalize some of these results to the clique complexes of these graphs. Furthermore, we determine which player has a winning strategy for some classes of threshold graphs.
\end{abstract}

\maketitle

\section{Introduction}\label{introduction}

Let $P$ be a partially ordered set with a global minimum $0$. In the \emph{game of chomp} on $P$ (also know as \emph{poset game}), two players $A$ and $B$ alternatingly pick an element of $P$ with $A$ being the first player. Whoever is forced to pick $0$ loses the game. A move consists of picking an element $x \in P$ and removing its {\it up-set}, that is,
all the elements that are larger or equal to $x$. The questions one wants to answer for a given $P$ are: \begin{center}

Has either of the players a winning strategy? Can a strategy be devised explicitly?                                                                                                         \end{center}

An easy and well-known observation with respect to the first of these questions is the following:

\begin{remark}\label{ganaAmax} If $P$ is a finite poset with a global maximum $1$, then player $A$ has a winning strategy. This can be proved with an easy (non-constructive) strategy stealing argument.
Indeed, if $A$ starting with $1$ cannot be extended to a winning strategy, then  $B$ has a devastating reply $x \in P$. But in this case, $A$ wins starting with $x$.
\end{remark}

One of the most well-known and probably oldest games that is an instance of chomp is Nim~\cite{Bouton}, where $P$ consists of a disjoint union of chains plus a global minimum. The first formulation in terms of posets is due to Schuh~\cite{Schuh}, where the poset is that of all divisors of a fixed number $N$, with $x$ below $y$ when $y|x$. A popular special case of this is the chocolate-bar-game introduced by Gale~\cite{Gale}, where $P$ is a finite grid. Another variant is to play chomp on the Boolean lattice, i.e., the inclusion order on all subsets of an $n$-element set. It was conjectured by Gale and Neyman in the 80s~\cite{GN}, that here taking the maximum element is always a good first move. After this was verified for $n\leq 6$ in the 90s~\cite{CT}, it was shown that the conjecture fails for $n=7$~\cite{BC}.
Recently, chomp was studied in \cite{chompsemigroups} for infinite posets arising from numerical semigroups and several algebraic properties could be used to establish winning strategies.
There is a rich body of research on chomp with respect to different classes of posets. For more information on the game and its history, we refer to~\cite{BCG,web,fr}.

This paper concerns chomp on (finite) simplicial complexes partially ordered by inclusion and we mostly investigate the chomp game on graphs, where the graph is regarded as a simplicial complex. Hence, the players take turns to remove either an edge or a vertex (and all its incident edges), and the player who cannot move because the remaining graph is the empty graph, loses. The game of chomp on graphs has been studied in \cite{KY,O}. In \cite{KY} the authors provide the Nim-value of bipartite graphs, complete multipartite graphs (see Theorem \ref{compmultipartito}), some families of pseudotrees, and state some conjectures concerning the Nim-values of pseudotrees. In \cite{O}, the author proves some of these conjectures and obtains the Nim-values of some other families of graphs, including some wheels and fans. Moreover, the (simple) pseudo-forests that are second-player win are characterized.

{\bf Our results.} After introducing some basics, in Section~\ref{sec:genKne}  we provide an explicit formula for the Nim-values of generalized Kneser graphs (Theorem \ref{th:genKneser}).  Also in this section we are able to decide which player has a winning
strategy for the chomp game on the clique complex of generalized Kneser graphs (Corollary \ref{th:genKneserclique}). In Section~\ref{sec:John} we study Johnson graphs and provide a formula for their Nim-values under certain hypotheses (Propositions \ref{nkpares} and \ref{n2k}). In Proposition \ref{noinvJohnson} we provide a negative result showing that our methods cannot be pushed forward to compute the Nim-value of every Johnson graph. Whenever we are able to obtain the Nim-value of a Johnson graph we prove that this value equals the one of its clique complex (Proposition \ref{cliqueJohnson}). 
Finally, in Section~\ref{sec:threshold} we study certain families of threshold graphs.

Whenever we are able to decide the outcome of the chomp game, we are also able to devise an explicit winning strategy. The sole exception is Corollary \ref{th:genKneserclique}, whose proof relies on the non-constructive strategy stealing argument of Remark \ref{ganaAmax}, and we do not know an alternative constructive proof.

We finish the paper with some concluding remarks in Section~\ref{sec:conc}.

\subsection{A little notation}

When playing chomp on finite posets, the game finishes after a finite number of moves and, since there are no draws, one of the players has a winning strategy (this is a particular case of the classical Zermelo's theorem in game theory, see, e.g., \cite{Zermelo}). For a given poset $P$ with a minimum, we denote $\ch(P) \in\{A,B\}$ as follows: $\ch(P) = A$ if the first player to move on $P$ has a winning strategy, and $\ch(P) = B$ otherwise. When $\ch(P) = A$, we say that $P$ is a {\it winning position}, otherwise we say that $P$ is a {\it losing position}.

Given $P_1$ and $P_2$ two posets with a global minimum each, we denote by $P_1 \cup_0 P_2$ the poset obtained by identifying both minima and without any extra comparability. Whenever $\ch(P_1) = B$, then it is easy to check that $\ch(P_1 \cup_0 P_2) = \ch(P_2)$. However, when $\ch(P_1) = \ch(P_2) = A$, then $\ch(P_1 \cup_0 P_2)$ can be either $A$ or $B$. To handle this situation and to be able to decide which player has a winning strategy on $P_1 \cup_0 P_2$, it is convenient to introduce the concept of Nim-value of a poset $P$, denoted by ${\rm Nim}(P)$. This value is defined inductively as follows: ${\rm Nim}(P) = 0$ if $P = \{0\}$, or $${\rm Nim}(P) =  {\rm mex}\{{\rm Nim}(P_x) \, \vert \, x \in P \setminus \{0\}\} \in \N,$$ where $P_x$ denotes the poset obtained after removing the up-set of $x \in P$, and for a finite $A \subset \N$,  ${\rm mex}(A)$ denotes the minimum excluded value of $A$, i.e., the smallest value $n \in \N$ such that $n \notin A$. From this definition it is not difficult to see that \begin{center} $\ch(P) = B$ if and only if ${\rm Nim}(P) = 0$. \end{center} Indeed, if $P$ is a poset with ${\rm Nim}(P) = 0$ and one player plays on $P$, the resulting poset is $P_x$ for some $x \in P$ and, by definition, ${\rm Nim}(P_x) \neq 0$. Conversely, if ${\rm Nim}(P) \neq 0$,  there exists an element $x \in P$ such that ${\rm Nim}(P_x) = 0$, so it suffices to choose such an $x$ to devise a winning strategy. The main interest of knowing the Nim-value is explained in the following classic result of Sprague and Grundy.

\begin{theorem}\cite{gru,spr}  \label{th:nimsum}Let $P_1$ and $P_2$ be two posets with a unique minimum. Then,
$${\rm Nim}(P_1 \cup_0 P_2) = {\rm Nim}(P_1) \oplus {\rm Nim}(P_2),$$ where $a \oplus b$ is the nonnegative integer whose binary encoding is the binary XOR operation on the binary encodings of $a$ and $b$.
\end{theorem}

\section{Chomp on generalized Kneser graphs}\label{sec:genKne}

The goal of this section is to compute the Nim-value of generalized Kneser graphs.

\begin{definition} For every triplet $(n,k,l) \in \N \times \Z^2$, the generalized Kneser graph $\KG(n,k,l)$ is the graph whose vertices correspond to the $k$-element subsets of the set $[1,n] := \{1,\ldots,n\}$, and where two vertices are adjacent if and only if the two corresponding sets intersect in at most $l$ elements (see Figure \ref{fig:petersen} for an example).
\end{definition}
\begin{figure}
\includegraphics[scale=.6]{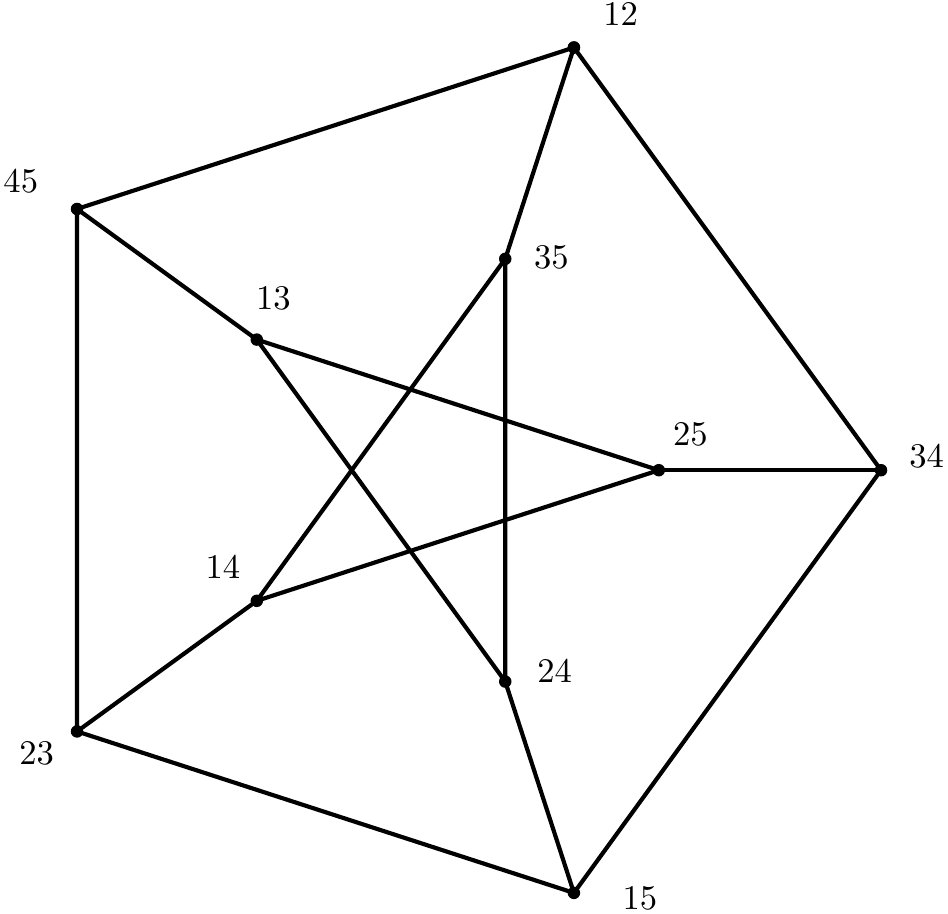}
\caption{$\KG(5,2,0)$ is the Petersen graph.}\label{fig:petersen}
\end{figure}

Classical Kneser graphs correspond to $l = 0$ in this definition. For convenience we have defined generalized Kneser graphs for all triplets $(n,k,l) \in \N \times \Z^2$; however, by definition we have the following.

\begin{lemma}\label{trivial} Let $(n,k,l) \in \N \times \Z^2$. Then,
\begin{itemize} \item[(a)] $\KG(n,k,l)$ is the empty graph if and only if $k < 0$ or $n < k$.
\item[(b)] $\KG(n,k,l)$  has no edges if and only if it is the empty graph, $l < 0$ or $n < 2k-l$.
\end{itemize}
\end{lemma}

The Nim-value of the empty graph is $0$ and, when a graph has no edges, its Nim-value is either $0$ or $1$, and coincides with the number of vertices modulo $2$. So, when $\KG(n,k,l)$ has no edges, we trivially get that  ${\rm Nim}(\KG(n,k,l)) = 1$ if and only if $\KG(n,k,l)$ is not the empty graph and $\binom{n}{k}$ is odd; and ${\rm Nim}(\KG(n,k,l)) = 0$ otherwise. Another easy remark is that whenever $l \geq k-1$, then $\KG(n,k,l) = \KG(n,k,k-1)$ and this coincides with $\mK_{\binom{n}{k}}$, the complete graph with $\binom{n}{k}$ vertices. So, there is no loss of generality in assuming that $l < k$.

All these easy considerations yield that it suffices to study the Nim-value of the graphs $\KG(n,k,l)$ when $(n,k,l) \in \N^3$ and $l < k$. The following result, which is the main result of this section, provides the Nim-value in all these
cases.

\begin{theorem}\label{th:genKneser}Let $(n,k,l) \in \N^3$ with $l < k$  and set $m := \lceil \log_2(k-l) \rceil$. Then,

	\[ {\rm Nim}(\KG(n,k,l)) =  \left( \binom{n\ {\rm mod}\ 2^m}{k\ {\rm mod}\ 2^m} \ {\rm mod}\ 2 \right) \cdot \left( {\lfloor n/2^m \rfloor  \choose \lfloor k /2^m \rfloor}\ {\rm mod} \ 3 \right). \]

\end{theorem}

To prove this result we are going to demonstrate that the Nim-value of a generalized Kneser graph coincides with the one of a complete multipartite graph, and then we will conclude by applying the following result from \cite{KY}.

\begin{theorem}\label{compmultipartito}\cite[Theorem 2]{KY} Let $\mK_{n_1,\ldots,n_r}$ denote the complete $r$-partite graph with partitions of sizes $n_1,\ldots,n_r \in \Z^+$. Then,
\begin{center} ${\rm Nim}(\mK_{n_1,\ldots,n_r}) = (t\ {\rm mod} \ 3)$; where $t := |\{i \, \vert \, n_i$ is odd$\}|$. \end{center}
\end{theorem}

The tool we use is to exploit the many symmetries of generalized Kneser graphs. In particular, we are going to use the following lemma.

\begin{lemma}\label{puntosfijos} \cite[Lemma 2.1]{O} Let $G = (V(G),E(G))$ be a graph and let $\varphi: V(G) \rightarrow V(G)$ be an automorphism of $G$ such that
\begin{itemize}
\item[(a)] $\varphi \circ \varphi = {\rm id}_V$ (i.e., $\varphi$ is an involution), and
\item[(b)]  $\{u, \varphi(u)\} \notin E$ for all $u \in V$.
\end{itemize}
Then, ${\rm Nim}(G) = {\rm Nim}(H)$, where $H = (V(H),E(H))$ is the induced subgraph of $G$ whose vertices are the fixed points of $\varphi$, i.e., $V(H) =\{u \in V(G) \, \vert \, \varphi(u) = u\}$.
\end{lemma}

Let us illustrate this result with an example. Consider the graph $G$ of Figure \ref{fig:involution} and set $\varphi$ the involution of $G$ sending $u_i \mapsto v_i$, $v_i \mapsto u_i$ and $w \mapsto w$ for $w \notin \{u_1,u_2,v_1,v_2\}$; then, the subgraph induced by the fixed points of $\varphi$ is $\mK_5$ and, by Lemma \ref{puntosfijos} and Theorem \ref{compmultipartito},  one gets that ${\rm Nim}(G) = {\rm Nim}(\mK_5) = 2$.

\begin{figure}
\includegraphics[scale=.6]{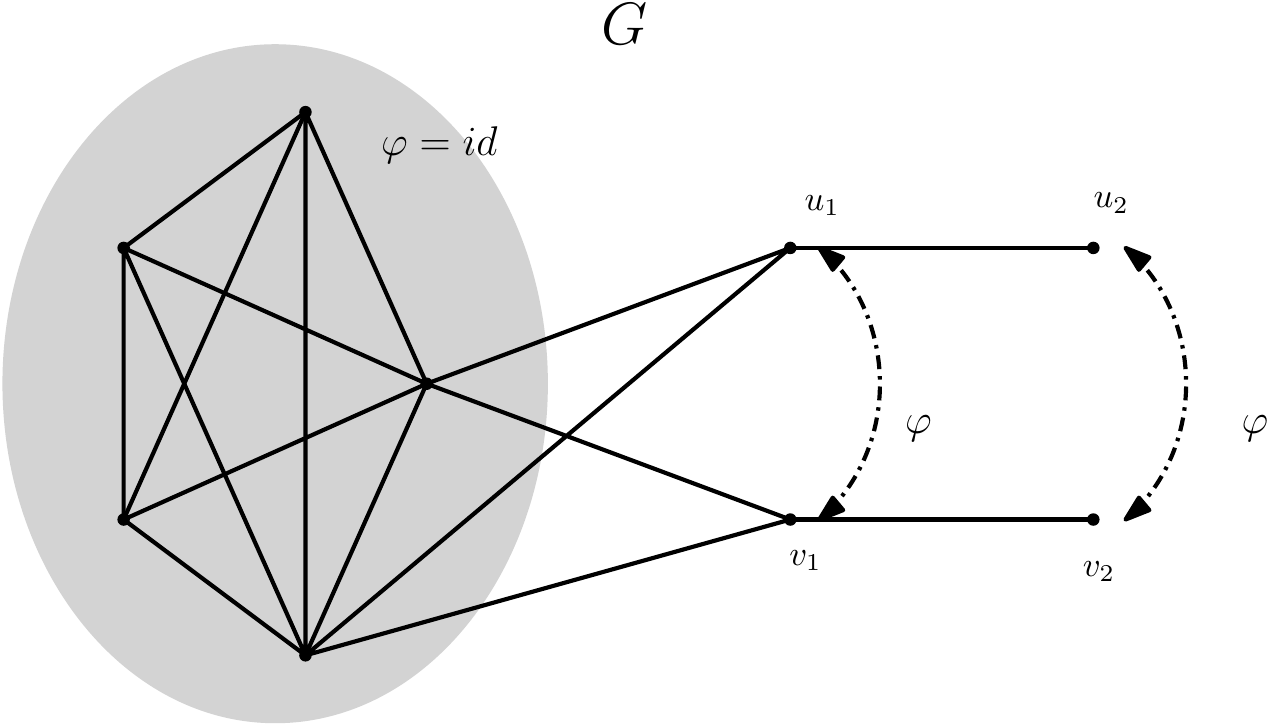}
\caption{${\rm Nim}(G) = {\rm Nim}(\mK_5) = 2$.}
\end{figure}\label{fig:involution}

An interesting feature of Lemma \ref{puntosfijos} is that it does not only reduce the computation of the Nim-value of $G$ to the one of the smaller graph $H$, but also exhibits a winning strategy provided one knows a winning strategy for $H$. The idea is the following, if ${\rm Nim}(H) = 0$, then player $B$ wins on chomp on $H = (V(H), E(H))$, thus whenever $A$ plays:
\begin{itemize}
\item $u \notin V(H)$, then $B$ answers $\varphi(u) \notin V(H)$,
\item $\{u,v\} \notin E(H)$,  then $B$ answers $\{\varphi(u),\varphi(v)\} \notin E(H)$, and
\item a vertex or an edge in $H$, then $B$ plays the corresponding winning answer in $H$.
\end{itemize}
Analogously, if ${\rm Nim}(H) > 0$ and player $A$ wins on chomp on $H = (V(H), E(H))$, then $A$ plays the winning move in $H$ and continues playing as described before.

As a direct consequence of Lemma \ref{puntosfijos}, we have the following result that applies for join graphs. Given to graphs $G_1 = (V_1,E_1)$ and $G_2 = (V_2,E_2)$,  the {\it join graph} of $G_1$ and $G_2$, which we denote by $G_1 + G_2$ is the graph with vertices $V_1 \cup V_2$ and
edges $E_1 \cup E_2 \cup (V_1 \times V_2)$.

\begin{corollary}\label{joingraph}Let $G = G_1 + G_2$ be the join graph of $G_1 = (V_1,E_1)$ and $G_2 = (V_2,E_2)$ and, for $i = 1,2$, let $\varphi_i: V_i \rightarrow V_i$ be an
automorphism of $G_i$ such that
\begin{itemize}
\item[(a)] $\varphi_i \circ \varphi_i = {\rm id}_{V_i}$, and
\item[(b)] $\{u, \varphi_i(u)\} \notin E_i$ for all $u \in V_i$.
\end{itemize}
Then, ${\rm Nim}(G) = {\rm Nim}(G')$, where $G' = H_1+ H_2$, and $H_i$ is the induced subgraph of $G_i$ with vertices are the fixed points of $\varphi_i$, i.e., $V(H_i) = \{u \in V_i \, \vert \, \varphi_i(u) = u\}$.
\end{corollary}
\begin{proof}It suffices to consider $\varphi: V_1 \cup V_2 \rightarrow V_2 \cup V_2$ the only automorphism extending $\varphi_1$ and $\varphi_2$.
It is straightforward to check that $\varphi$ satisfies the hypotheses of Lemma \ref{puntosfijos} and, as a consequence, the result follows.
\end{proof}

In this section we will consider several times a particular type of involution of induced subgraphs of $\KG(n,k,l)$ which is given by a permutation $\pi \in \Sigma_n$ of order $2$; this  type of involutions is described in the following technical lemma:

\begin{lemma}\label{permutacion}Let $k \in [0,n]$ and let also
\begin{itemize} \item $H = (V(H),E(H))$ be an induced subgraph of $\KG(n,k,l)$, and
\item $\pi \in \Sigma_n$ be a permutation of order $2$, i.e.,  $\pi = (a_1,a_2) \cdots (a_{2t-1},a_{2t})$ for some different values $a_1,\ldots,a_{2t} \in [1,n].$ \end{itemize} Consider $\varphi_{\pi} \in {\rm Aut}(\KG(n,k,l))$ defined as $\varphi_{\pi}(\{b_1,\ldots,b_k\}) := \{\pi(b_1),\ldots,\pi(b_k)\}.$ If
 \begin{itemize}\item[(a)] $\varphi_{\pi}(V(H)) = V(H)$, and
 \item[(b)] $t < k-l$; \end{itemize} then ${\rm Nim}(H) = {\rm Nim}(H')$, where $H'$ is the
 induced subgraph of $H$ with vertex set $$V(H') := \{u \in V(H) \, \vert \, a_{2i+1} \in u \Leftrightarrow a_{2i+2} \in u \text{ \ for\  all\ }i \in [0,t-1] \}.$$
\end{lemma}
\begin{proof}The morphism $\varphi_{\pi}:  V(H) \longrightarrow V(H)$ is well defined, it is an involution (because $\pi$ is of order $2$) and  $\varphi_{\pi}$ is an automorphism
of $H$ (because $|u \cap v| = |\varphi_{\pi}(u) \cap \varphi_{\pi}(v)|$ for all $u,v \in V(H)$ and $H$ is an induced subgraph of $\KG(n,k,l)$).  Moreover, $\{u, \varphi_{\pi}(u)\} \notin E(H)$ because $|u \cap \varphi_{\pi}(u)| \geq k - t > l.$  It suffices to observe that $$\{u \in V(H) \, \vert \, \varphi_{\pi}(u) = u \} = \{u \in V(H) \, \vert \, a_{2i+1} \in u \Leftrightarrow a_{2i+2} \in u, \ \forall i \in  [0,t-1]\}$$ and apply Lemma \ref{puntosfijos} to get the result.
\end{proof}

The following lemma will be crucial to prove our result.

\begin{lemma}\label{kneserlema} Let $m \in \Z^+$ such that $2^m < {\rm min}(2(k - l), n + 1)$ and set $I_m := [1,2^m] = \{1,\ldots,2^m\}$; then
$${\rm Nim}(\KG(n,k,l)) = {\rm Nim}(H_m),$$
where $H_m$ is the induced subgraph of $\KG(n,k,l)$ with vertex set
$$ V(H_m) :=  \{u \in V(\KG(n,k,l)) \ \vert \ I_m \subseteq u {\text \ or \ }  I_m \cap u = \emptyset\}. $$
\end{lemma}
\begin{proof}We set $M := 2^m$ and consider iterative applications of Lemma \ref{permutacion} with the following permutations $\pi_1,\ldots, \pi_{m}$:
\[ \begin{array}{llll}
\pi_1 & := & (1 ,\,2) (3 ,\,4) (5 ,\,6) \cdots  (M - 1 ,\,M), \\
\pi_2 & := & (1 ,\,3) (2 ,\,4) (5 ,\,7) \cdots  (M - 2 ,\,M), \\
& & \ldots \\
\pi_i & := & \prod_{j \in \{1,\ldots,M\} \atop{0 < (j\ {\rm mod}\ 2^{i}) \leq 2^{i-1}}} (j ,\, j + 2^{i-1}), \\
& & \ldots \\
\pi_{m} & := & (1 ,\, \frac{M}{2} + 1) (2 ,\, \frac{M}{2} + 2) \cdots  (\frac{M}{2},\, M). \\
\end{array}
 \]
We observe that each of the $\pi_i$'s is a product of $M/2 = 2^{m-1} < k-l$ transpositions and, thus, we may apply Lemma \ref{permutacion}.
As a consequence of
iteratively applying Lemma \ref{permutacion} with $\pi_1,\ldots,\pi_{m-1}$ we get that the fixed vertices after the application of $\pi_i$ for all $i$
are those of $H_m$; hence ${\rm Nim}(\KG(n,k,l)) ={\rm Nim}(H_m)$ and the result follows.
\end{proof}

%
%
%

As a consequence of the previous Lemma we get the following result. This result relates the Nim-value of a generalized Kneser graph with the one of the join graph of two smaller generalized Kneser graphs.

\begin{proposition}\label{kneserdescomposicion}Let $(n,k,l) \in \N^3$ such that $2 \leq k-l$ and take $m \in \Z^+$ such that $k-l \leq 2^m < 2(k-l)$. If $2^m \leq n$, then:
$${\rm Nim}(\KG(n,k,l)) = {\rm Nim}(\KG(n-2^m, k, l) + \KG(n-2^m, k-2^m, l-2^m)).$$
\end{proposition}
\begin{proof}By Lemma \ref{kneserlema}, we have that ${\rm Nim}(\KG(n,k,l)) = {\rm Nim}(H_m),$ with $H_m$ is the induced subgraph of $\KG(n,k,l)$ with vertex set $$V(H_m) :=  \{u \in V(\KG(n,k,l)) \ \vert \ I_m \subseteq u {\text \ or \ }  I_m \cap u = \emptyset\};$$ where $I_m = [1,\ldots,2^m]$. Taking $G_1$ the induced subgraph with vertex set $$V(G_1) :=  \{u \in V(\KG(n,k,l)) \ \vert \ I_m \subseteq u\}$$ it turns out that $G_1 \simeq \KG(n-2^m,k-2^m,l-2^m)$, and taking $G_2$ the induced subgraph with vertex set $$V(G_2) :=  \{u \in V(\KG(n,k,l)) \ \vert \ I_m \cap u = \emptyset\}$$ it turns out that $G_2 \simeq \KG(n-2^m,k,l)$. Finally, we observe that for all $u \in V(G_1),\, v \in V(G_2)$ we have that $\{u,v\} \in E(H_m)$ because $|u \cap v| \leq k - 2^m \leq l$.
\end{proof}

\bigskip

The general idea in the proof of Theorem \ref{th:genKneser} is to first apply Proposition \ref{kneserdescomposicion} to get that the Nim-value of a generalized Kneser graph equals the Nim-value of the join of two (smaller) generalized Kneser graphs. Then, by Corollary \ref{joingraph}, we may apply iteratively Proposition \ref{kneserdescomposicion} to these smaller graphs until the number of vertices of the resulting graphs is $< 2^m$. Hence, we stop this procedure when the resulting graph is a join of several graphs of the form $\KG(n-t2^m, k - i2^m, l - i 2^m)$ for some $0 \leq i \leq t$ and $n -t2^m < 2^m$. It turns out that all these smaller graphs do not have edges and, thus, the resulting graph is a complete multipartite graph. Therefore we derive that the Nim-value of a generalized Kneser graph coincides with that of a certain complete multipartite graph, and apply Theorem \ref{compmultipartito} to
conclude the result.

\medskip

\noindent {\it Proof of Theorem \ref{th:genKneser}.} If $l = k-1$, then $m = 0$. Moreover, $\KG(n,k,l)$ equals the complete graph on $\binom{n}{k}$ vertices and, by  Theorem \ref{compmultipartito}, we have that ${\rm Nim}(\KG(n,k,l)) = (\binom{n}{k} \ {\rm mod} \ 3)$. Hence the result holds for $l = k-1$.

From now on, we assume that $l \leq k-2$. If $n \geq 2^m$, we apply Proposition \ref{kneserdescomposicion} and get that $${\rm Nim}(\KG(n,k,l)) = {\rm Nim}(\KG(n-2^m, k, l) + \KG(n-2^m, k-2^m, l-2^m)).$$
If $n - 2^m \geq 2^m$, we can apply again Proposition \ref{kneserdescomposicion} together with Corollary \ref{joingraph} and get that $$\begin{array}{llll} {\rm Nim}(\KG(n,k,l)) =  {\rm Nim}( &  \KG(n - 2 \cdot 2^m, k, l) + \\ & 2 \cdot \KG(n-2 \cdot 2^m, k-2^m, l-2^m) + \\ & \KG(n-2 \cdot 2^m, k-2\cdot 2^m, l-2 \cdot 2^m)).\end{array};$$
where $2 \cdot G$ denotes the join graph $G + G$ and, in general, $m \cdot G$ denotes $(m-1) \cdot G + G$ for $m \in \N$. If we set $t := \lfloor n/2^m \rfloor$ and repeat this argument we get that $${\rm Nim}(\KG(n,k,l)) = {\rm Nim}\left(\sum_{i = 0}^t \binom{t}{i} \cdot \KG(n-t \, 2^m, k-i  2^m, l-i 2^m)\right).$$
Moreover, for each $i \in \{0,\ldots,t\}$, by Lemma \ref{trivial}, we have that $\KG(n-t \, 2^m, k-i 2^m, l-i 2^m)$ is the empty graph if and only if $n - t2^m < k - i 2^m$ or $k - i 2^m < 0$; thus, it
only remains to consider the values $i$ such that $$\frac{k-n}{2^m} + t \leq i \leq \frac{k}{2^m}.$$
If there is no integer value in the interval $[\frac{k-n}{2^m} + t, \frac{k}{2^m}]$, then ${\rm Nim}(\KG(n,k,l)) = {\rm Nim}(\emptyset) = B$ and $\binom{n - \lfloor n/2^m \rfloor \,2^m}{k - \lfloor k/2^m \rfloor 2^m} = \binom{n\ {\rm mod}\ 2^m}{k\ {\rm mod}\ 2^m} = 0$, which is even and the result follows. If $[\frac{k-n}{2^m} + t, \frac{k}{2^m}] \cap \N \not= \emptyset$ then it has only one element $j := \lfloor k/2^m \rfloor$. So, $${\rm Nim}(\KG(n,k,l)) = {\rm Nim}\left(\binom{t}{j} \cdot \KG(n-t  2^m, k-j 2^m, l-j 2^m)\right).$$
We claim that $\KG(n-t 2^m, k-j 2^m, l-j 2^m)$ has no edges. Indeed, if we set $n' := n-t 2^m$, $k' :=  k-j 2^m$ and $l' := l-j 2^m$, we observe that if $l' < 0$, then $\KG(n',k',l')$ has no edges and if $l' \geq 0$, we have that the following inequalities hold
\[ 2 k' - l' \geq 2 (k' - l') = 2(k-l) > 2^m > n';\]
 thus, by Lemma \ref{trivial}, $\KG(n',k',l')$ has no edges. Hence, the resulting graph a complete $\binom{t}{j}$-partite graph and each of the partitions has  $\binom{n-t 2^m}{k-j 2^m}$ vertices.  Now the result follows from Theorem \ref{compmultipartito}. If $\binom{n-t 2^m}{k-j 2^m}$ is even, then ${\rm Nim}(\KG(n,k,l)) = 0$. Otherwise, it is $\binom{t}{j}\ ({\rm mod}\ 3)$. \hfill $\qed$

\bigskip

\begin{corollary}\label{whowinsgenKneser}Let $(n,k,l) \in \N^3$ with $l < k$  and set $m := \lceil \log_2(k-l) \rceil$. Then,

	\[ {\rm Nim}(\KG(n,k,l)) = A \textrm{\  if\ and\ only\ if \ } 2 \nmid \binom{n\ {\rm mod}\ 2^m}{k\ {\rm mod}\ 2^m} {\text \ and\ }3 \nmid {\lfloor n/2^m \rfloor  \choose \lfloor k /2^m \rfloor}. \]

\end{corollary}

\bigskip

Given $p$ a prime number, a classical result by Kummer \cite{kummer} states that the biggest power of $p$ that divides $\binom{a+b}{a}$ (also called the $p$-valuation) coincides with the number of carries when $a$ and $b$ are added in base $p$. Also a nice result by Lucas \cite{lucas} provides a formula to compute the value of a binomial coefficient modulo $p$. Both results have as corollary that $p$ divides $\binom{a+b}{a}$ if and only if there is a carry when $a$ and $b$ are added in base $p$.  Thus one could restate Theorem \ref{th:genKneser} in terms of the binary and ternary encoding of $n,k$ and $l$.

\subsection{Chomp on the clique complex of $\KG(n,k,l)$} \label{cliquecomplex}
\

Given a graph $G = (V(G),E(G))$, we recall that a clique of $G$ is a subset $V' \subset V(G)$ such that the edge $\{u,v\} \in E(G)$ for all distinct $u, v \in V'$. We denote by $\mC(G)$ the {\it clique complex} of $G$, that is, the simplicial complex with vertex set $V(G)$, and whose faces are the
cliques of $G$. Moreover, for  $s \in \N$, we denote by $\mC_s(G)$ the {\it $(s-1)$-skeleton of $\mC(G)$}, that is, the simplicial complex whose vertex set is $V(G)$, and whose faces are the cliques of $G$ of size $\leq s$. We observe that $\mC(G) = \mC_s(G)$ for all $s \geq |V(G)|$ (or for any $s$ greater or equal to the clique number of the graph).
Since $\mC_0(G)$ is the empty simplicial complex and the only non-empty faces of $\mC_1(G)$ are the vertices of the graph, the interesting cases are when $s \geq 2$.
The faces of any simplicial complex are partially ordered by inclusion, hence, it makes sense to play chomp in $\mC_s(G)$ with $s \in \N$.

The following result, which has the same hypotheses of Lemma \ref{puntosfijos}, reduces the computation
of ${\rm Nim}(C_s(G))$ to the computation of the Nim-value of the clique complex of an induced subgraph of $G$. 

\begin{lemma}\label{puntosfijosclique} Let $s \in \N$, $G = (V(G),E(G))$ be a graph and $\varphi \in {\rm Aut}(G)$ such that
\begin{itemize}
\item[(a)] $\varphi \circ \varphi = {\rm id}_V$, and
\item[(b)]  $\{u, \varphi(u)\} \notin E$ for all $u \in V$.
\end{itemize}
Then, ${\rm Nim}(\mC_s(G)) = {\rm Nim}(\mC_s(H))$, where $H = (V(H),E(H))$ is the induced subgraph of $G$ whose vertices are the fixed points of $\varphi$, i.e., $V(H) =\{u \in V(G) \, \vert \, \varphi(u) = u\}$.
\end{lemma}

Lemma \ref{puntosfijosclique} is a generalization of Lemma \ref{puntosfijos} and a direct consequence of the following one, which appeared in a preliminary version of \cite{fr}. Since this result disappeared
in the final version of \cite{fr}, we include here their original proof.

\begin{lemma} Let $(P,\leq)$ be a finite poset and let $\psi:  P \rightarrow P$ such that
\begin{itemize}
\item[(a)] $\psi \circ \psi = {\rm id}_P$,
\item[(b)]  $x \leq y$ if and only if $\psi(x) \leq \psi(y)$, and
\item[(c)] the subposet $F$ of $P$ with vertices $\{x \in P \, \vert \, \psi(x) = x\}$ (the fixed points of $\psi$), is a down-set.
\end{itemize}
Then, ${\rm Nim}(P) = {\rm Nim}(F)$.
\end{lemma}
 \begin{proof}Let $F'$ be a copy of the poset $F$ disjoint from $P$. By Theorem \ref{th:nimsum},  ${\rm Nim}(P) = {\rm Nim}(F)$  if and only if ${\rm Nim}(P \cup_0 F') = 0$. Moreover, this is equivalent to check that  $\ch(P \cup_0 F') = B$. To prove the result we are going to devise a winning strategy for player $B$ on $P \cup_0 F'$. The strategy is the following:  whenever $A$ picks an element in $F$ (respect. in $F'$), then $B$ picks the same element in $F'$ (respect. in $F$) and, whenever $A$ picks an element $y \in P \setminus F$, then $B$ chooses $\psi(y)$. Since $F$ is a down-set, this latter pair of moves does not disturb $F$ or $F'$. Since $P$ is finite, player $A$ will eventually be forced to pick $0$ and, hence, this is a winning strategy for $B$.
 \end{proof}

An easy application of Lemma \ref{puntosfijosclique} yields the following result (which generalizes Theorem \ref{compmultipartito}).

\begin{proposition}\label{compmultipartitoclique} Let $s \in \N$, then,
\begin{center} ${\rm Nim}(\mC_s(\mK_{n_1,\ldots,n_r})) = {\rm Nim}(\mC_s(\mK_t))$; where $t := |\{i \, \vert \, n_i$ is odd$\}|$. \end{center}
\end{proposition}

Thus, if one follows the lines of the proof of Theorem \ref{th:genKneser} replacing Lemma \ref{puntosfijos} by Lemma \ref{puntosfijosclique} and Theorem \ref{compmultipartito} by Proposition \ref{compmultipartitoclique}, one gets  the following result:

\begin{theorem}\label{th:genKnesercliquetrunc}Let $s \in \N$, $(n,k,l) \in \N^3$ with $l < k$  and set $m := \lceil \log_2(k-l) \rceil$. Then,

	\[ {\rm Nim}(\mC_s(\KG(n,k,l))) =  \left( \binom{n\ {\rm mod}\ 2^m}{k\ {\rm mod}\ 2^m} \ {\rm mod}\ 2 \right) \cdot {\rm Nim}(\mC_s(\mK_t)), \]
	where $t = {{\lfloor n/2^m \rfloor}  \choose {\lfloor k /2^m \rfloor}}$.

\end{theorem}

So, the computation of the Nim-value of the $(s-1)$-skeleton of the clique complex of any generalized Kneser graph reduces to the one of ${\rm Nim}(\mC_s(\mK_t))$ for some $t \in \N$.
The value of ${\rm Nim}(\mC_s(\mK_t))$ is not known in general. When $s \leq 2$, by Theorem \ref{compmultipartito}, we have that ${\rm Nim}(\mC_s(\mK_t)) = (t \ {\rm mod}\ (s+1))$. For any value of $s \in \N$, Gale and Neyman \cite{GN} conjectured that ${\rm Nim}(\mC_s(\mK_t)) = 0$ if and only if $s+1$ divides $t$, however, Brauer and Christensen \cite{BC} disproved this conjecture by showing that the first player loses in the chomp game on $\mC_3(\mK_7)$ and, thus, $\mC_3(\mK_7) = 0$. When $s \geq t$, the simplicial complex $\mC_s(\mK_t) = \mC(\mK_t)$ has a maximum (the set of all vertices of the graph), and the 'strategy stealing' argument of Remark \ref{ganaAmax} yields $\ch(\mC(\mK_t)) = A$ for all $t \geq 1$ and, hence, ${\rm Nim}(\mC(\mK_t)) > 0$. Even if we know that $\ch(\mC(\mK_t)) = A$ for all $t \geq 1$, the problem of finding an explicit winning strategy for this simplicial complex is still open. Indeed, it was conjectured by Gale and Neyman \cite{GN} that taking the maximum is the (unique) winning move and proved in the same paper that it this is true for $n \leq 5$, and later by Christensen and Tilford \cite{CT} for $n = 6$. Recently, Brauer and Christensen \cite{BC} have proved that this is no longer true for $n = 7$, where the winning move is to take a set of $4$ elements.

As a consequence of Theorem \ref{th:genKnesercliquetrunc} and the fact that $\ch(\mC(\mK_t)) = A$ if and only if $t \geq 1$, we have the following result which characterizes which
player has a winning strategy in the clique complex of any generalized Kneser graph:

\begin{corollary}\label{th:genKneserclique}Let $(n,k,l) \in \N^3$ with $l < k \leq n$  and set $m := \lceil \log_2(k-l) \rceil$. Then,

	\[ \ch(\mC(\KG(n,k,l))) =  A \ \Longleftrightarrow \  \binom{n \ {\rm mod}\  2^m}{k \ {\rm mod}\ 2^m} {\text \ is \ odd}.\]

\end{corollary}

\section{Chomp on Johnson graphs}\label{sec:John}

In this short section we will study the chomp game on Johnson graphs.

\begin{definition} For every  pair $(n,k) \in \N^2$ with $0 \leq k \leq n$, the Johnson graph $\J(n,k)$ is the graph whose vertices correspond to the $k$-element subsets of a set of $n$ elements, and where two vertices are adjacent if and only if the two corresponding sets intersect in exactly $k-1$ elements (see Figure \ref{fig:johnson} for an example).
\end{definition}
\begin{figure}
\includegraphics[scale=.6]{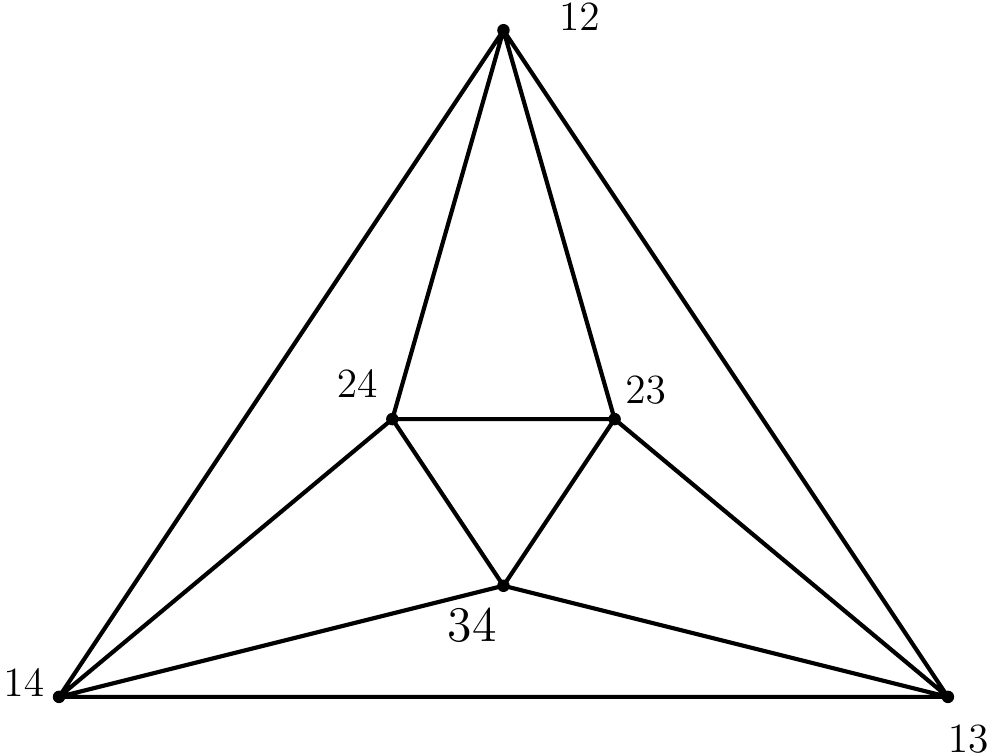}
\caption{The Johnson graph $\J(4,2)$.}
\end{figure}\label{fig:johnson}

For Johnson graphs, we will follow a similar strategy as in the previous section and see how we can exploit the symmetries of Johnson graphs to obtain the Nim-value when either both $n$ and $k$ are even (Proposition \ref{nkpares}), and when $n = 2k$ (Proposition \ref{n2k}). In order to show that the same strategy cannot be used to obtain the Nim-value of any Johnson graph, we will see in Proposition \ref{noinvJohnson} that for all other pairs $(n,k)$, there are no nontrivial involutions satisfying the hypotheses of Lemma \ref{puntosfijos}. In fact, for these pairs we do not know which player has a winning strategy. The smallest such Johnson graph  is $\mathcal{J}(5,2)$ -- the complement of the Petersen graph.

\begin{proposition}\label{nkpares}Let $(n,k) \in \N^2$ with $0 \leq k \leq n$. If $k$ and $n$ are even, then:
\[ {\rm Nim}(\J(n,k)) = \binom{n}{k} \ {\rm mod} \ 2. \]
\end{proposition}
\begin{proof}Set $\pi := (1,\ 2)(3, \ 4) \cdots (n-1,\  n) \in \Sigma_n$ and consider $\varphi_{\pi}$ the map on the vertices of  $\J(n,k)$ defined as $$\varphi_{\pi}(\{b_1,\ldots,b_k\}) := \{\pi(b_1),\ldots,\pi(b_k)\}.$$
Since $\pi$ is a permutation of order $2$, then $\varphi_{\pi}$ is an involution. Moreover,  $\varphi_{\pi}$ is an endomorphism of $G$ (indeed, it is an isomorphism). Moreover, $\{u, \varphi_{\pi}(u)\}$ is never an edge of $\J(n,k)$ because by construction $|u \cap \varphi_{\pi}(u)|$ is even and $k-1$ is odd. Thus, the hypotheses of Lemma \ref{puntosfijos} are satisfied and ${\rm Nim}(\J(n,K)) = {\rm Nim}(H)$, where $H$ is the subgraph induced by the fixed points of $\varphi(\pi)$. We observe that for any two vertices $u,v$, if $u = \varphi_{\pi}(u)$ and $v = \varphi_{\pi}(v)$, then $|u \cap v|$ is even and, hence, $\{u,v\}$ is not an edge. Since $H$ is a graph without edges and with $\binom{n/2}{k/2}$ vertices, ${\rm Nim}(H)$ equals $\binom{n/2}{k/2}\ {\rm mod} \ 2$ . To finish the proof it suffices to observe that  $\binom{n/2}{k/2}$ and $\binom{n}{k}$ have the same parity because both $n$ and $k$ are even.
\end{proof}

\begin{proposition}\label{n2k}Let $k \in \N$, then \[ {\rm Nim}(\J(0,0))  = 1,\, {\rm Nim}(\J(2,1)) = 2	 {\text \ and \ } {\rm Nim}(\J(2k,k)) = 0 {\text \ for \ all \ } k \geq 2.\]
\end{proposition}
\begin{proof}Since $\J(0,0) = \mathcal K_1$ and $\J(2,1) = \mK_2$, ${\rm Nim}(\J(0,0)) = 1 $ and ${\rm Nim}(\J(2,1)) = 2$.
Assume now that $k \geq 2$. Set $\varphi$ the involution on the set of vertices of $\J(2k,k)$ defined as $\varphi(u) := \{1,\ldots,2k\} \setminus \{u\}$. We observe that $\{u, \varphi(u)\}$ is never an edge and that $\varphi$ has no fixed points. Thus, by Lemma \ref{puntosfijos}, we conclude that ${\rm Nim}(\J(2k,k)) = {\rm Nim}(\emptyset) = 0$.
\end{proof}

The proofs of propositions \ref{nkpares} and \ref{n2k} consist of applying Lemma \ref{puntosfijos} to get a graph without edges. Thus, the Nim-value only depends on
the parity of the number of vertices of the resulting graph, which coincides with the parity of the number of vertices of the original Johnson graph. If we consider the clique complex of the Johnson graph or its $(s-1)$-skeleton (see Subsection \ref{cliquecomplex}), and use Lemma \ref{puntosfijosclique} in an analogous way, we get the following result.

\begin{corollary}\label{cliqueJohnson}Let $(n,k) \in \N^2$ with $0 \leq k \leq n$. If $k$ and $n$ are even or $n = 2k$, then ${\rm Nim}(\mC_s(\J(n,k))) = {\rm Nim}(\J(n,k))$ for all $s \geq 1$.
\end{corollary}

In the proofs of Proposition \ref{nkpares} and Proposition \ref{n2k}, we have considered the following automorphisms of Johnson graph's:
\begin{enumerate} \item the relabeling map $\varphi_{\pi}(\{b_1,\ldots,b_k\}) = \{\pi(b_1),\ldots,\pi(b_k)\}$, where $\pi \in \Sigma_n$ is a permutation, and
\item  when $n = 2k$, the complementation map $\sigma(u) = [1,n] \setminus u$.
\end{enumerate}
Indeed, according to \cite{BCN} (see also \cite[Theorem 2]{Jones}), the group of automorphisms of $\J(n,k)$ is exactly $\{\varphi_{\pi} \, \vert \, \pi \in \Sigma_n\} \simeq \Sigma_n$ when $n \neq 2k$, and$\{\varphi_{\pi} \, \vert \, \pi \in \Sigma_n\} \cup \{\varphi_{\pi} \circ \sigma \, \vert \, \pi \in \Sigma_n\}$ when $n = 2k$.

\begin{proposition}\label{noinvJohnson} If $n \neq 2k$ and either $n$ or $k$ is odd. Then, the only involution of $\J(n,k)$ satisfying the hypotheses of Lemma \ref{puntosfijos} is the identity.
\end{proposition}
\begin{proof} When $n \neq 2k$, the automorphisms that are also involutions are $\{\varphi_{\pi}\, \vert \, \pi$ is a permutation of order $2\}$. If $\varphi_{\pi}$ is not the identity, we may assume without loss of generality that $\pi = (1, \ 2) \cdots (2t-1,\ 2t)$ for some $t \geq 1$. For $k$ odd, we consider the vertex $u := \{1,3,\ldots,k+1\}$ and observe that $\{u, \varphi_{\pi}(u)\}$ is an edge because $\varphi_{\pi}(u) = \{2,3,\ldots,k+1\}$; hence $\varphi_{\pi}$ does not satisfy the hypotheses of Lemma \ref{puntosfijos}.
For $k$ even and $n$ odd, we consider the vertex $u := \{1,n-k+1,\ldots,n\}$ and observe that $\{u, \varphi_{\pi}(u)\}$ is an edge because $\varphi_{\pi}(u) = \{2,n-k-1,\ldots,n\}$; again here we conclude that $\varphi_{\pi}$ does not satisfy the hypotheses of Lemma \ref{puntosfijos}.
\end{proof}

\section{Chomp in some subfamilies of threshold graphs}\label{sec:threshold}
In this section we study the chomp game in some subfamilies of  threshold graphs. Given nonnegative integers $n,k$ and $i_j\le n$, with $j=1,\ldots,k$, we define the graph $\mK_n^{i_1,\ldots,i_k}$ as the graph containing a
 clique $\mK_n$ with vertices $\{u_{1},\ldots,u_{n}\}$ and $k$ new vertices, $\{v_{i_1},\ldots,v_{i_k}\}$, such that $v_{i_j}$ is adjacent to $\{u_{1},\ldots,u_{i_j}\}$  for every $1\le j\le k$. This family of graphs is known as the family of \emph{threshold graphs}.

We prove which player has a winning strategy for chomp on  $\mK_n^{i}$ for $n\ge i\ge0$ (see Table $(b)$) and for chomp on $\mK_n^{j,i}$ for some values on $i$ and $j$ (see Tables $(c),(d)$ and $(e)$).
\begin{figure}[htb]
\begin{center}
 \includegraphics[width=.8\textwidth]{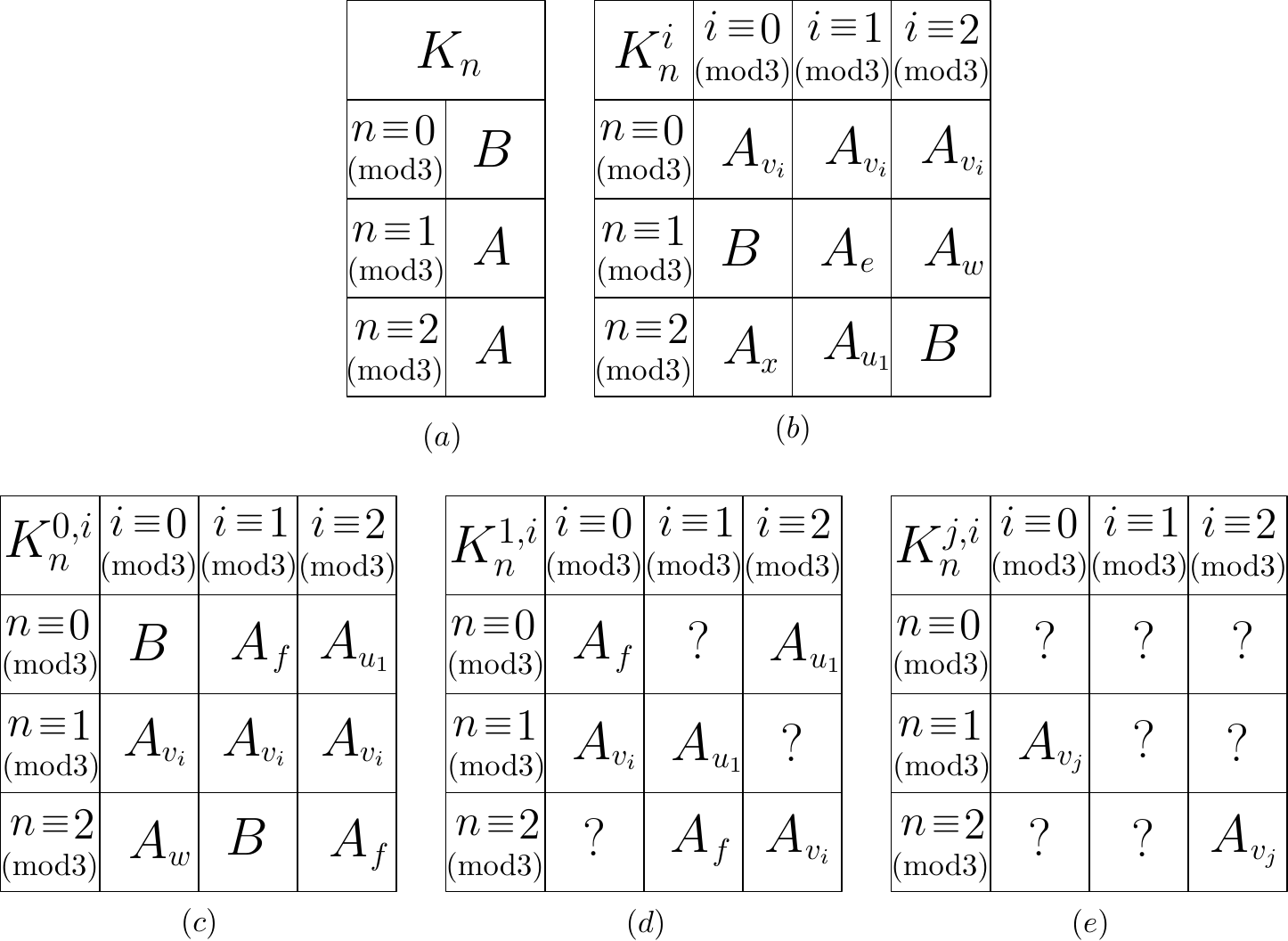}
 \caption{The subscript in each letter $A$ indicates which element is the winning movement for player $A$ where $e=u_1v_i$, $f=u_1v_1$, $x$ is a vertex of $\mK_n$ different from $u_1,\ldots,u_i$ and $w$ is an edge of $\mK_n$ not incident to the vertices $u_1,\ldots,u_i$. The symbol ``?'' denotes that the problem is open in this case. }\label{Tabla}
\end{center}
\end{figure}
In order to prove these results, we use the following lemma that proves which player has a winning strategy for chomp in complete graphs and is a direct consequence of 
Theorem \ref{compmultipartito}.
\begin{lemma}[{\cite[Lemma 1]{KY}}]\label{completas}
   Player $B$ has a winning strategy for chomp on $\mK_n$ when $n\equiv 0$ $({\rm mod}\  3)$, and loses otherwise (see Table $(a)$).
\end{lemma}
We first prove which player has a winning strategy for chomp on $\mK_n^{i}$ for every $i\ge0$ when $n=i,i+1,i+2$.
\begin{lemma}\label{nfijo}
 Let $i\ge0$ and $n=i,i+1,i+2$. Then player $B$ has a winning strategy for chomp on $\mK_i^n$ when $i, n \equiv 2$ $({\rm mod}\  3)$, when $i, n-1 \equiv 0$ $({\rm mod}\  3)$ and loses otherwise.
\end{lemma}
\begin{proof} If $n=i$ then $\mK_i^i=\mK_{i+1}$. If $n=i+1$, we notice that chomp in $\mK_{i+1}^i$ is equivalent to chomp in $\mK_i$ since  there exists an involution on vertices $v_i$ and the vertex of $\mK_n$ not adjacent to $v_i$. Then, these cases can be solved by Lemma \ref{completas}. If $n=i+2$, we have the following cases:

\textbf{\emph{Case $i\equiv 0$ $({\rm mod}\  3)$.}} Player $A$ wins by removing  a vertex $x$ different from $u_1,\dots,u_i$ since player $A$ gives to  player $B$ the graph $\mK_{i+1}^i$. As chomp in $\mK_{i+1}^i$ is equivalent to chomp in $\mK_i$, player $A$ wins by Lemma \ref{completas}.

\textbf{\emph{Case $i\equiv 1$ $({\rm mod}\  3)$.}} Player $A$ wins by removing vertex $v_{i}$ since player $A$ gives to  player $B$ the graph $\mK_{i+2}$. Then player $A$ wins by Lemma \ref{completas}.

\textbf{\emph{Case $i\equiv 2$ $({\rm mod}\  3)$.}} Player $A$ wins by removing an edge $w=xy$ of $\mK_n$ not incident to vertices $u_1,\ldots,u_i$. As there exists an involution on vertices $x$ and $y$  in $\mK_{i+2}^i-w$, player $A$ gives to  player $B$ the graph $\mK_{i+1}$. Then player $A$ wins by Lemma \ref{completas}.
\end{proof}
Next, we prove which player has a winning strategy for chomp on $\mK_n^{i}$ for every $n\ge i\ge0$ and $i=0,1,2$.
\begin{lemma}\label{ifijo}
Let  $n\ge i\ge0$ and $i=0,1,2$. Then player $B$ has a winning strategy for chomp on  $\mK_n^0$ when $n\equiv 1$ $({\rm mod}\  3)$, for chomp on $\mK_n^2$ when $n\equiv 2$ $({\rm mod}\  3)$ and loses otherwise.
\end{lemma}
\begin{proof}
 For each $i=0,1,2$, we proceed by induction on $n$ where the base cases are $n=0,1,2$. For $n=0$, $\mK_0^{i}=\mK_1$ and player  $A$ wins. For $n=1$, player $B$ wins when $i=0$ since $\mK_1^{0}$ consists of two isolated vertices and  player $A$ wins otherwise by deleting the edge $u_1v_i$. For $n=2$, player $B$ wins when $i=2$ by Lemma \ref{completas} and player $A$ wins otherwise by deleting vertex $u_1$. Now suppose the lemma holds for $i=0,1,2$ and every positive integer $j\le n-1$. Let us consider the graph $\mK_n^i$ for $n\ge3$ and  $0\le i\le2$.
By induction hypothesis and by Lemma \ref{completas}, one can check that player $A$ wins in the cases when appears the letter $A$ in Table $(b)$  by removing the element that the subscript indicates. Next, we prove that player $B$ wins for the two remaining cases.

First suppose that $n\equiv 1$ $({\rm mod}\  3)$ and $i=0$. Then player $A$ loses by removing any element. If player $A$ removes vertex $v_i$, $A$ loses by Lemma \ref{completas}. If player $A$ removes any vertex of $\mK_n$, then $A$ loses by induction hypothesis. If player $A$ removes some edge $w=xy$, then  there exists an involution on vertices $x$ and $y$ in $\mK_n^0-w$, giving to player $B$ the graph $\mK_{n-2}^0$. As $n-2\equiv 2$ $({\rm mod}\  3)$ then player $B$ wins by induction hypothesis. Therefore, player $A$ loses  in any case.

Now suppose that $n\equiv 2$ $({\rm mod}\  3)$ and $i=2$. Then player $A$ loses by removing any element.
One can check that for all cases except one, player $A$ gives to  player $B$ a graph (or a graph who is equivalent in chomp to a graph) corresponding to a letter $A$ in Tables $(a)$ or $(b)$  with fewer vertices. Thus, the result holds by  induction hypothesis for almost all the cases. The special case is when player $A$ removes an edge $f=xu_2$ of $\mK_n$ with $x\not\in\{u_1,u_2\}$. In this case the graph $\mK_{n}^{2}-f$ does not have involutions and is not a graph of Tables $(a)$ or $(b)$. Nevertheless, player $B$ can remove vertex $v_2$. Hence, there exists an involution on vertices $x$ and $u_2$ in $\mK_n^2-f-v_2$, giving to  player $A$ the graph $\mK_{n-2}$. Hence, player  $A$ loses by Lemma \ref{completas}.
\end{proof}
We are now able to prove the next theorem.
\begin{theorem}\label{general}
Let $n\ge i\ge 0$. Then,
 \begin{itemize}
   \item [(i)]  player $B$ has a winning strategy for chomp on $\mK_n^i$ when $n\equiv 1$ $({\rm mod}\  3)$ and $i\equiv 0$ $({\rm mod}\  3)$, when $n\equiv 2$ $({\rm mod}\  3)$ and $i\equiv 2$ $({\rm mod}\  3)$, and  loses otherwise (see Table $(b)$).
   \item [(ii)] player $B$ has a winning strategy for chomp on $\mK_n^{0,i}$ when $n,i\equiv 0$ $({\rm mod}\  3)$, when $n\equiv 2$ $({\rm mod}\  3)$ and $i\equiv 1$ $({\rm mod}\  3)$, and  loses otherwise (see Table $(c)$).
   \item [(iii)] player $A$ has a winning strategy for chomp on $\mK_n^{1,i}$ when $n,i\equiv 0$ $({\rm mod}\  3)$, when $n\equiv 0$ $({\rm mod}\  3)$ and $i\equiv 2$ $({\rm mod}\  3)$, when $n\equiv 1$ $({\rm mod}\  3)$ and $i\equiv 0$ $({\rm mod}\  3)$, when $n,i\equiv 1$ $({\rm mod}\  3)$, when $n\equiv 2$ $({\rm mod}\  3)$ and $i\equiv 1$ $({\rm mod}\  3)$, and when $n,i\equiv 2$ $({\rm mod}\  3)$ (see Table $(d)$).
   \item [(iv)] player $A$ has a winning strategy for chomp on $\mK_n^{j,i}$ for every $j\ge0$ when $n\equiv 1$ $({\rm mod}\  3)$ and $i\equiv 0$ $({\rm mod}\  3)$, and when $n,i\equiv 2$ $({\rm mod}\  3)$  (see Table $(e)$).
 \end{itemize}
\end{theorem}
\begin{proof} (i). We prove by induction on $n$ and $i$.  Lemma  \ref{nfijo} proves the cases $n=i,i+1,i+2$ and every $i\ge0$. Lemma \ref{ifijo} proves the cases $i=0,1,2$ and every $n\ge i$. Now suppose the theorem holds for  $\mK_n^{j}$ when $j\le i-1$ and $n\ge j$, and also for $\mK_m^{i}$ when  $m\le n-1$ and $i\ge0$. Let us consider the graph $\mK_n^i$ with $i\ge3$ and $n\ge i+3$.
By induction hypothesis and by Lemma \ref{completas}, one can check that player $A$ wins whenever we have a letter $A$ in Table $(b)$  by removing the element indicated by the subscript. Next, we prove that player $B$ wins for the two remaining cases.

When $n\equiv 1$ $({\rm mod}\  3)$ and $i\equiv 0$ $({\rm mod}\  3)$, player $A$ loses by removing any element.  One can check that for all cases except one, player $A$ gives to  player $B$ a graph (or a graph with the same chomp value) corresponding  to a letter $A$ in Tables $(a)$ or $(b)$ with fewer vertices. Hence, the result holds by  induction hypothesis. The special case is when player $A$ removes an edge $f=xu_j$,  for some $j\in\{1,\ldots,i\}$, where  $x\not\in\{u_1,\ldots,u_i,v_i\}$. In this case, the graph $\mK_{n}^{i}-f$ does not have involutions and is not a graph consider in Tables $(a)$ or $(b)$. Nevertheless, player $B$ can remove the edge $e=v_iu_j$. Then, there exists an involution on vertices $x$ and $u_j$ in $\mK_n^i-f-e$, giving to player $A$ the graph $\mK_{n-2}^{i-1}$.  As $n-2,i-1\equiv 2$ $({\rm mod}\  3)$, player  $A$ loses also in this case by induction hypothesis. Similar arguments occur  when $n,i\equiv 2$ $({\rm mod}\  3)$.

(ii) and (iii). Table $(c)$  shows which player wins for chomp on  $\mK_n^{0,i}$ for every $n\ge i\ge 0$ and Table $(d)$ shows, for some values of $n$ and $i$, the winning positions of player $A$ of chomp on  $\mK_n^{1,i}$.  We omit this proof since the ideas are similar to the proofs of Lemmas \ref{nfijo}, \ref{ifijo} and to (i). We notice that the subscript in each letter $A$  of Tables $(c)$ and $(d)$ indicates which element is a winning movement for player $A$ because the resulting graph corresponds to a position with letter $B$ in Table $(b)$  or $(c)$.

(iv). For every $j\ge0$, we notice that player $A$ wins when $n\equiv 1$ $({\rm mod}\  3)$ and $i\equiv 0$ $({\rm mod}\  3)$ by removing vertex $v_j$, since player $A$ gives to player $B$ a graph corresponding  to a letter $B$ in  Table $(b)$. The same occurs when $n,i\equiv 2$ $({\rm mod}\  3)$.
\end{proof}

We have not been able to complete all the cases in Tables $(d)$ and $(e)$.  For instance, we do not know the value of ${\rm Chomp}(\mK_n^{1,i})$ for all possible values 
$n \equiv 0 \, ({\rm mod}\ 3)$ and $i \equiv 1 \, ({\rm mod}\ 3)$.
Nevertheless, taking $i = 1$, we have that ${\rm Chomp}(\mK_n^{1,1}) = B$ when $n\equiv 0$ $({\rm mod}\  3)$ because $\ch(\mK_n^{1,1})=\ch(\mK_n)$.
We wonder if player $B$ has a winning strategy for chomp on $\mK_n^{1,i}$ when $n\equiv 0$ $({\rm mod}\  3)$ and $i\equiv 1$ $({\rm mod}\  3)$.

Similarly, for $\mK_n^{1,2}$ and $n\equiv 1$ $({\rm mod}\  3)$, we know that player $A$ wins by removing the vertex $u_2$. We do not know if
player $A$ has a winning strategy for chomp on $\mK_n^{1,i}$ when $n\equiv 1$ $({\rm mod}\  3)$ and $i\equiv 2$ $({\rm mod}\  3)$.

For chomp on $\mK_n^{1,i}$ when $n\equiv 2$ $({\rm mod}\  3)$ and $i\equiv 0$ $({\rm mod}\  3)$, we have a different situation. Player $B$ wins when $i=0$ (see Table $(c)$) but loses when $i=3$ since player $A$ can remove edge $u_2u_3$, the resulting graph is $\mK_{n-2}^{1,1}$ and $\ch(\mK_{n-2}^{1,1})=\ch(\mK_{n-2})$. Hence, one may conclude that $\ch(\mK_n^{1,i})$  does not depend on the parameters of the graph modulo $3$ when $n\equiv 2$ $({\rm mod}\  3)$ and $i\equiv 0$ $({\rm mod}\  3)$. However, we do not know what happens when the parameters $1,i$ are in increasing order, i.e., for $i\geq 1$?. 

More generally, we wonder the following: for $i_1\leq \ldots\leq i_k$, does the outcome of chomp on the threshold graph $\mK_n^{i_1,\ldots,i_k}$
can be described only in terms of the parameters $i_1,\ldots,i_k,n$  modulo $3$?


\section{Conclusions}\label{sec:conc}
We have enlarged the graph families for which Nim-values or winning-strategies for chomp are known. The strongest results could be obtained in graphs (and their clique complexes) with symmetries such as generalized Kneser graphs and certain Johnson graphs. However, our method cannot be used for some families of Johnson graphs. We think the next class to attack here are Johnson graphs of the form $\mathcal{J}(2k+1,2)$, i.e., line graphs of odd complete graphs.


Finally, we determined the chomp value of some threshold graphs. We observed some patterns depending only on the parameters of the graph modulo three.  Could it be true that for $i_1\leq \ldots\leq i_k$, the value of $\ch(\mK_n^{i_1,\ldots,i_k})$ only depends on $(n \mod 3, i_1\mod 3,\ldots, i_k \mod 3)$?


\section*{Acknowledgements}

L.P.M. has been supported by the Mexican National Council on Science and Technology (C\'atedras-CONACYT). K.K. has been supported by ANR projects GATO: ANR-16-CE40-0009-01, DISTANCIA: ANR-17-CE40-0015, and CAPPS: ANR-17-CE40-0018. I. G. has been supported by Ministerio de
Econom\'ia y Competitividad,  Spain (MTM2016-78881-P). 

This research was initiated on a visit of I.G. and K.K. at CIMAT supported by UMI Laboratoire Solomon Leftschetz - LaSol - no. 2001  CNRS-CONACYT-UNAM, Mexique. Both I.G. and K.K. 
thank L.P.M. and Mireia Ferrer for their hospitality and the good time spent during the stay in Guanajuato.

Finally, we thank Cormac O'Sullivan for pointing out an error in a theorem about almost bipartite graphs stated in an earlier version of the paper.  



\begin{thebibliography}{}

\bibitem{BCG} E. R. Berlekamp, J. H. Conway, R. K. Guy. Winning Ways for Your Mathematical
Plays. 2nd ed. A K Peters, Ltd., Wellesley, MA (2001).


\bibitem{Bouton}
C. L. Bouton.  Nim, a game with a complete mathematical theory.
Annals of Mathematics, 3:35--39, (1902).

\bibitem{web} A.~E. Brouwer. The game of Chomp, \texttt{https://www.win.tue.nl/$\sim$aeb/games/chomp.html}.

\bibitem{BC} A. E. Brouwer and J. D. Christensen. Counterexamples to Conjectures About Subset Takeaway and Counting Linear Extensions of a Boolean Lattice. ArXiv prepring 	arXiv:1702.03018 [math.CO].

\bibitem{BCN} A. E. Brouwer, A. M. Cohen, A. Neumaier. Distance-regular graphs. Ergebnisse der Mathematik und ihrer Grenzgebiete (3) [Results in Mathematics and Related Areas (3)], 18. Springer-Verlag, Berlin, 1989.


\bibitem{CT} J. D. Christensen, M. Tilford. David Gale's subset take-away game, Amer. Math. Monthly 104 (1997) 762-766.

\bibitem{fr}
S.~A. Fenner, J. Rogers. Combinatorial game complexity: an introduction with poset games. Bull. Eur. Assoc. Theor. Comput. Sci. EATCS No. 116 (2015), 42--75.

\bibitem{Gale} D. Gale. A curious Nim-type game. Amer. Math. Monthly 81 (1974) 876-879.

\bibitem{GN} D. Gale, A. Neyman. Nim-type games, Internat. J. Game Theory 11 (1982) 17-20.

\bibitem{chompsemigroups}
I. Garc\'ia-Marco, K. Knauer. Chomp on numerical semigroups. Accepted in Algebraic Combinatorics. Arxiv preprint ArXiv:1705.11034 [math.CO], 2017.

\bibitem{gru} P. M. Grundy. Mathematics and games. Eureka, 2:6--8, 1939.

\bibitem{Jones} G. A. Jones. Automorphisms and regular embeddings of merged Johnson graphs. European Journal
of Combinatorics, 26:417-435, 2005.

\bibitem{KY} T. Khandhawit, L. Ye. Chomp on graphs and subsets, Arxiv prerpint arXiv:1101.2718 [math.CO], 2011.

\bibitem{kummer} E. Kummer. \"Uber die Erg\"anzungss\"atze zu den allgemeinen Reciprocit\"atsgesetzen. Journal f\"ur die reine und angewandte Mathematik. 44: 93--146, 1852.

\bibitem{lucas} E. Lucas. Th\'eorie des Fonctions Num\'eriques Simplement P\'eriodiques. American Journal of Mathematics. 1 (2): 184--196, (3): 197-- 240, (4): 289--321; 1878.

\bibitem{O} C. O'Sullivan. A vertex and edge deletion game on graphs. Arxiv preprint arXiv:1709.01354 [math.CO], 2017.

\bibitem{Schuh} F. Schuh. Spel van delers. Nieuw Tijdschrift voor Wiskunde 39 (1952) 299--304.

\bibitem{spr} R. P. Sprague. \"Uber mathematische Kampfspiele. Tohoku Mathematical Journal,
41:438--444, 1935--1936.

\bibitem{Zermelo} U. Schwalbe, P. Walker. Zermelo and the Early History of Game Theory. Games and Economic Behavior 34(1):123-137, 2001.


\end{thebibliography}
\end{document}